\documentclass[12pt]{amsart}
\usepackage{amsmath,amsthm,amssymb,mathtools}
\usepackage{enumitem}
\usepackage[textheight=8.8in, textwidth=6.75in]{geometry}
\usepackage{microtype}
\usepackage[hidelinks]{hyperref}

\numberwithin{equation}{section}
\newtheorem{theorem}{Theorem}[section]
\newtheorem{proposition}{Proposition}
\newtheorem{corollary}[theorem]{Corollary}
\newtheorem{lemma}[theorem]{Lemma}
\theoremstyle{definition}
\newtheorem{definition}{Definition}

\theoremstyle{remark}
\newtheorem{remark}{Remark}

\usepackage{soul}
\usepackage{xcolor}

\title[Chebyshev quotients and Demazure multiplicities]{Chebyshev quotients, Demazure multiplicities, and Dyck-path models}
\author{Rekha Biswal, Ken Ono, and Jujian Zhang}

\address{$^{1}$ School of Mathematical Sciences\\
National Institute of Science Education and Research (NISER)\\
Bhubaneswar 752050, India}

\address{$^{2}$ Homi Bhabha National Institute\\
Training School Complex\\
Anushakti Nagar, Mumbai 400094, India}

\email{rekhabiswal27@gmail.com}
\email{rekha@niser.ac.in}

\address{Axiom Math, 124 University Avenue, Palo Alto, CA 94301}
\email{ken@axiommath.ai}
\email{jujian@axiommath.ai}

\date{}

\begin{document}

\begin{abstract}
We study Chebyshev quotients that arise in the representation theory of Lie algebras, specifically within the theory of Demazure flags for fusion products of
 $\mathfrak{sl}_2[t]$-modules.  Using a recent formula that expresses numerical Demazure multiplicities as
coefficients of such quotients, we prove a general eventual non-negativity
theorem for the same rational functions that compute these multiplicities:
each quotient either terminates or has strictly positive coefficients for
sufficiently large degrees, which we in turn interpret in terms of matchings
and bounded walks.
In several natural infinite
families, these are unsigned bounded Dyck path models, giving both a structural explanation for the observed positivity phenomenon
and concrete combinatorial models for key families of Demazure multiplicities.
The theorems in this paper were autonomously produced and formalized in Lean/Mathlib
 by AxiomProver from natural-language statements.
\end{abstract}

\maketitle

\section{Introduction}

Fusion products for $\mathfrak{sl}_2[t]$, in the sense of Chari and Venkatesh, give a
representation-theoretic framework for many finite-dimensional graded current-algebra modules,
including local Weyl modules and Demazure modules
\cite{BiswalChariSchneiderViswanath2016,ChariVenkatesh2015,FeiginLoktev1999}.  Let
\[
\xi=(\xi_1\ge \cdots \ge \xi_\ell>0)
\]
be a partition, and write $|\xi|=\xi_1+\cdots+\xi_\ell$.  We denote by $V(\xi)$ the
corresponding fusion product.  If $m\ge \xi_1$, then $V(\xi)$ admits a level-$m$ Demazure
flag, also called an {\it excellent filtration}, and the multiplicities of the graded shifts of
$D(m,n)$ are independent of the chosen flag (see, for example,
\cite{ChariSchneiderShereenWand2014}).

The graded multiplicity polynomials arising from these filtrations encode refined structural
information about the modules.  They also interact with tensor product decompositions of
irreducible integrable highest weight modules over affine Lie algebras through character
identities, filtrations, and recursive structures; see \cite{JakelicMoura2018}.  Explicit
combinatorial interpretations of their coefficients can therefore make these multiplicities more
transparent and more computable.

Following the notation of \cite{Biswal2026}, we write
\begin{equation}\label{eq:graded-multiplicity}
V_n^{\xi\to m}(q)=\sum_{p\ge 0}[V(\xi):\tau_p^*D(m,n)]\,q^p
\end{equation}
for the associated graded multiplicity polynomial.  The present paper concerns the numerical
specialization $q=1$.  The Chebyshev-quotient formula of \cite{Biswal2026}, recalled in
Proposition~\ref{prop:cheb-quotient}, shows that $V_n^{\xi\to m}(1)$ is obtained by extracting a
single coefficient from a rational function built from a simple Chebyshev-type polynomial
sequence.  Our aim is to analyze these rational functions directly: we prove an eventual
positivity dichotomy for their coefficient sequences, give an exact signed combinatorial formula
for those coefficients, and isolate families in which the signs cancel to give bounded Dyck-path
models.

The novelty of the present paper is not the Chebyshev-quotient formula itself, which is recalled from \cite{Biswal2026}, but rather the systematic analysis of the coefficient sequences that arise from that formula.  We prove that, after the evident cancellations, these quotients satisfy a sharp dichotomy: they either terminate or are eventually strictly positive. 
We then explain the coefficients by an explicit signed model involving matchings and bounded walks, and identify natural infinite families for which the quotient factors into Dyck-path-compatible pieces, yielding direct unsigned bounded-Dyck-path interpretations.
Thus, this paper turns the Chebyshev quotient from a compact character-theoretic expression into a structural and combinatorial description of numerical Demazure multiplicities. In this way, the results refine the Demazure-flag multiplicity formulas arising from the
Chari--Venkatesh framework by identifying the eventual sign behavior and by producing
explicit path models for natural infinite families of numerical multiplicities.

We use the polynomial sequence
\[
p_0(x)=p_1(x)=1,\qquad p_{r+1}(x)=p_r(x)-x p_{r-1}(x)\quad (r\ge 1),
\]
and, for a partition $\xi=(\xi_1\ge \cdots \ge \xi_\ell)$, set
\begin{equation}\label{eq:pxi}
p_\xi(x):=\prod_{i=1}^\ell p_{\xi_i}(x).
\end{equation}
For a power series $f(x)=\sum_{r\ge 0} c_r x^r$ and an integer $j$, we use the convention
$[x^j]f(x)=c_j$ if $j\ge 0$, and $[x^j]f(x)=0$ otherwise.  Throughout,
$\binom{A}{B}=0$ unless $0\le B\le A$. The following fact was proved by the first author in earlier work (see Theorem~2.3 of \cite{Biswal2026}).

\begin{proposition}\label{prop:cheb-quotient}
Let $\xi$ be a partition, let $m\ge \xi_1$, and let $\mu\in \mathbb Z_{\ge 0}$.  Write
\[
\mu=\mu_1m+\mu_0,\qquad 0\le \mu_0<m.
\]
Then
\[
V_\mu^{\xi\to m}(1)
=
[x^{(|\xi|-\mu)/2}]
\frac{p_{m-\mu_0-1}(x)p_\xi(x)}{p_m(x)^{\mu_1+1}}.
\]
The coefficient is understood to be $0$ if $(|\xi|-\mu)/2$ is not a nonnegative integer.
\end{proposition}

Proposition~~\ref{prop:cheb-quotient} is the bridge from Demazure flags to the Chebyshev-polynomial side.
For fixed representation-theoretic data $(\xi,m,\mu)$, the numerical Demazure
multiplicity $V^{\xi\to m}_{\mu}(1)$ is exactly the coefficient of degree
$(|\xi|-\mu)/2$ in the quotient
\[
\frac{p_{m-\mu_0-1}(x)p_{\xi}(x)}{p_m(x)^{\mu_1+1}}.
\]
Thus, every statement about the coefficient sequence of this quotient is, in
particular, a statement about the Chebyshev expression that computes the
corresponding Demazure-flag multiplicity. Theorem~\ref{thm:main} should be understood in
this sense: it identifies the eventual coefficient behavior of precisely the
rational functions that occur in the Demazure multiplicity formula. Theorem~\ref{thm:comb}
then explains these coefficients by a signed matching-and-walk model, while
Theorem~\ref{thm:manifest} identifies natural families in which this signed model becomes an
unsigned bounded-Dyck-path model. Finally, Corollary~\ref{cor:manifest-mult} translates these
quotient identities back into direct formulas for numerical Demazure
multiplicities.

Theorem~1.1 is the structural result that underlies the rest of the paper.
It applies not to an auxiliary family of generating functions, but directly to
the Chebyshev quotients appearing in Proposition~1. Consequently, when
$m\geq \xi_1$, it describes the eventual coefficient behavior of the very
rational functions whose coefficient extractions compute the numerical
Demazure multiplicities $V^{\xi\to m}_{\mu}(1)$.

\subsection{Background, definitions, and statements}

We begin with the basic combinatorial objects used in the second theorem.

\begin{definition}
For an integer $r\ge 0$, let $P_r$ denote the path graph on vertices
$1,2,\dots,r$, with edge set
\[
\{\{i,i+1\}:1\le i<r\}.
\]
When $r=0$, we interpret $P_0$ as the empty graph.
A \emph{matching} in $P_r$ is a set of pairwise disjoint edges. We write $m_r(j)$ for the number
of matchings in $P_r$ having exactly $j$ edges.
\end{definition}

\begin{definition}
Fix $m\ge 1$ and integers $0\le a,b\le m-1$.
A \emph{strip walk of height $m-1$ from $a$ to $b$} is a finite sequence
\[
(h_0,h_1,\dots,h_L)
\]
of integers such that
\[
h_0=a,\qquad h_L=b,\qquad h_{i+1}-h_i\in\{\pm 1\}\ \text{for all }i,
\]
and
\[
0\le h_i\le m-1\qquad (0\le i\le L).
\]
We denote by $w_L^{(m)}(a,b)$ the number of such walks of length $L$.
When $a=0$ and $b=m-1$, we call these full-height strip walks.
\end{definition}

\begin{definition}
If $\gamma$ is a full-height strip walk of height $m-1$ and length $L$, then $L\equiv m-1\pmod 2$.
We define the \emph{excess} of $\gamma$ by
\[
e(\gamma):=\frac{L-(m-1)}{2}.
\]
Thus $e(\gamma)$ is a nonnegative integer.
More generally, if $\gamma$ is a strip walk from $a$ to $b$ with $0\le a\le b\le m-1$ and length $L$,
then $L\equiv b-a\pmod 2$, and we define
\[
e(\gamma):=\frac{L-(b-a)}{2}.
\]
Again $e(\gamma)$ is a nonnegative integer.
\end{definition}

\begin{definition}
A \emph{Dyck path} is a lattice path in $\mathbb Z^2$ starting at $(0,0)$, ending on the
$x$-axis, using up-steps $U=(1,1)$ and down-steps $D=(1,-1)$, and never going below the
$x$-axis; see, for example, \cite{Stanley2015}. Its \emph{semilength} is half of its total
number of steps. The \emph{height} of a Dyck path is the largest $y$-coordinate reached
along the path.
\end{definition}

Fix $m\ge 1$ and integers $0\le a,b\le m-1$ with $a+b\le m-1$. For $u\ge 0$, let
$\mathcal D_m(a,b;u)$ be the set of Dyck paths of height at most $m-1$ and semilength
$m-1-b+u$ whose first $a$ steps are up-steps and whose last $m-1-b$ steps are down-steps.
We write
\begin{equation}
D_m(a,b;u):=\#\mathcal D_m(a,b;u).
\end{equation}
In particular, $\mathcal D_m(0,0;u)$ is the set of Dyck paths of height at most $m-1$ and
semilength $m-1+u$ whose last $m-1$ steps are down-steps.

The first main result is the eventual-positivity dichotomy for these quotient coefficients.

\begin{theorem}\label{thm:main}
Fix $m\ge 1$, let $\xi$ be a partition with $\xi_i\le m$ for all $i$, and let
$\mu\in\mathbb Z_{\ge 0}$. Write
\[
\mu=\mu_1m+\mu_0,\qquad 0\le \mu_0<m,
\]
and set
\[
F(x):=F_{\xi,m,\mu}(x)=\frac{p_{m-\mu_0-1}(x)p_\xi(x)}{p_m(x)^{\mu_1+1}}=\sum_{r\ge 0}a_r x^r.
\]

If we let 
$t:=\#\{i: \xi_i=m\},$
then the following are true:
\smallskip

\noindent
(1) If $m=1$, then $F_{\xi,1,\mu}(x)=1$.

\noindent
(2) Assume that $m\ge 2$.
  \begin{enumerate}
  \item[{\rm {(a)}}] If $t\ge \mu_1+1$, then $F_{\xi,m,\mu}(x)$ is a polynomial. In particular,
  $a_r=0$ for all sufficiently large $r$.
  \item[{\rm {(b)}}] If $t\le \mu_1$, then $a_r>0$ for all sufficiently large $r$.
  \end{enumerate}
\end{theorem}

In representation-theoretic terms, Theorem~\ref{thm:main} separates the Chebyshev
quotients arising from Demazure flags into two cases. If sufficiently many
parts of $\xi$ are equal to the level $m$, then the denominator
$p_m(x)^{\mu_1+1}$ is cancelled and the quotient is a polynomial; consequently
only finitely many possible coefficient extractions can be nonzero. Otherwise
the quotient has a genuine pole at the smallest positive root of $p_m$, and
the resulting coefficient sequence is eventually strictly positive. Since
Proposition~1 identifies the relevant Demazure multiplicity with one
coefficient of this same quotient, Theorem~\ref{thm:main} gives a structural explanation
for why the Chebyshev expressions governing these multiplicities have eventual
positivity after the evident cancellations.

The next theorem gives the promised combinatorial description of the coefficients.
To state it, we use the notation of Theorem~\ref{thm:main}, and assume
\[
k:=\mu_1+1-t\ge 0.
\]
Let
\[
\alpha_0:=m-\mu_0-1,
\]
and let
\[
\alpha_1,\dots,\alpha_L
\]
be the parts of $\xi$ that are strictly smaller than $m$, listed with multiplicity. Thus
\[
p_{m-\mu_0-1}(x)\prod_{\xi_i<m}p_{\xi_i}(x)=\prod_{i=0}^L p_{\alpha_i}(x),
\]
and
\[
F_{\xi,m,\mu}(x)=\frac{\prod_{i=0}^L p_{\alpha_i}(x)}{p_m(x)^k}.
\]
For $u\ge 0$, we define
\[
B_m(u):=w_{m-1+2u}^{(m)}(0,m-1).
\]
In words, $B_m(u)$ is the number of full-height strip walks whose excess is $u$.

\begin{theorem}\label{thm:comb}
With the notation above, for every $r\ge 0$ one has
\[
a_r=
\sum_{\substack{j_0,\dots,j_L\ge 0\\ u_1,\dots,u_k\ge 0\\
 j_0+\cdots+j_L+u_1+\cdots+u_k=r}}
(-1)^{j_0+\cdots+j_L}
\left(\prod_{i=0}^L \binom{\alpha_i-j_i}{j_i}\right)
\left(\prod_{\nu=1}^k B_m(u_\nu)\right).
\]
Here, the second product is interpreted as $1$ when $k=0$.
Equivalently, $a_r$ is the signed count of tuples
\[
(M_0,M_1,\dots,M_L,\gamma_1,\dots,\gamma_k)
\]
with the following properties:
\begin{enumerate}
\item for each $0\le i\le L$, the object $M_i$ is a matching in the path graph $P_{\alpha_i}$;
\item for each $1\le \nu\le k$, the object $\gamma_\nu$ is a full-height strip walk of height $m-1$;
\item the total weight condition
\[
|M_0|+|M_1|+\cdots+|M_L|+e(\gamma_1)+\cdots+e(\gamma_k)=r
\]
holds;
\item the sign of such a tuple is
\[
(-1)^{|M_0|+|M_1|+\cdots+|M_L|}.
\]
\end{enumerate}
\end{theorem}

\begin{remark}
For $m\ge 2$, Theorem~\ref{thm:main} says that the only obstruction to strict eventual
positivity is whether enough copies of $p_m$ cancel to make the rational function $F(x)$
a polynomial. Theorem~\ref{thm:comb} explains why positivity is not obvious term-by-term:
in general the coefficients are not counting a single family of ordinary partitions or paths,
but rather a signed combination of matchings and bounded walks.
\end{remark}

The next result identifies families for which the quotient admits a factorization compatible with the denominator exponent $k$, leading to unsigned Dyck-path models.

\begin{theorem}
\label{thm:manifest}
Assume the notation in Theorem~\ref{thm:comb}, so that
\[
F_{\xi,m,\mu}(x)=\frac{\prod_{i=0}^L p_{\alpha_i}(x)}{p_m(x)^k}.
\]
Assume that there exist integers
\[
(a_\nu,b_\nu)\qquad (1\le \nu\le k)
\]
such that
\[
0\le a_\nu,b_\nu\le m-1,\qquad a_\nu+b_\nu\le m-1
\]
for every $\nu$, and such that
\[
\prod_{i=0}^L p_{\alpha_i}(x)=\prod_{\nu=1}^k p_{a_\nu}(x)p_{b_\nu}(x).
\]
Because $p_0(x)=p_1(x)=1$, this factorization condition depends only on the polynomial product
\[
\prod_{i=0}^L p_{\alpha_i}(x),
\]
rather than on the particular presentation of that product by the multiset
\[
\{\alpha_0,\dots,\alpha_L\}.
\]
However, the existence of a decomposition into exactly $k$ admissible pairs $(a_\nu,b_\nu)$ is itself a nontrivial restriction.
Then we have
\[
F_{\xi,m,\mu}(x)=\prod_{\nu=1}^k\left(\sum_{u\ge 0} D_m(a_\nu,b_\nu;u)\,x^u\right).
\]
In particular, $a_r$ is the number of $k$-tuples
\[
(P_1,\dots,P_k)
\]
in which each $P_\nu$ belongs to $\mathcal D_m(a_\nu,b_\nu;u_\nu)$ for some $u_\nu\ge 0$ and
\[
u_1+\cdots+u_k=r.
\]
Hence all coefficients $a_r$ are nonnegative.

\smallskip
\noindent
Moreover, this criterion produces the following explicit infinite families of unsigned quotient
identities.  In these three families, the quotient is $F_{\xi,m,\mu}(x)$ with $\mu=|\xi|$.
A notation such as $m^t$ records a multiset of parts; the corresponding partition is arranged
in nonincreasing order.
\begin{enumerate}
\item[\textup{(a)}] Let
\[
\xi=(m^t,1^s),
\]
and write
\[
s=qm+\rho,\qquad q\in\mathbb Z_{\ge 0},\quad 0\le \rho<m.
\]
Then, for every $N\ge 0$, $a_N$ counts tuples
\[
(P_0,P_1,\dots,P_q)
\]
with the following properties:
\begin{enumerate}[label=\textup{(\roman*)}]
\item $P_0\in \mathcal D_m(0,m-\rho-1;u_0)$, so $P_0$ is a Dyck path of height at most $m-1$
      and semilength $\rho+u_0$ whose last $\rho$ steps are down-steps;
\item for $1\le \nu\le q$, one has $P_\nu\in \mathcal D_m(0,0;u_\nu)$, so $P_\nu$ is a Dyck
      path of height at most $m-1$ and semilength $m-1+u_\nu$ whose last
      $m-1$ steps are down-steps;
\item $u_0+u_1+\cdots+u_q=N$.
\end{enumerate}

\item[\textup{(b)}] Let
\[
\xi=(m^t,r,1^s),
\qquad 1\le r\le m-1.
\]
Write
\[
r+s=qm+\rho,\qquad q\in\mathbb Z_{\ge 0},\quad 0\le \rho<m.
\]
If $q=0$, equivalently if $r+s<m$, then, for every $u\ge 0$, $a_u$ counts Dyck paths in
$\mathcal D_m(r,m-r-s-1;u)$.  Equivalently, $a_u$ is the number of Dyck paths of height at most
$m-1$ and semilength $r+s+u$ whose first $r$ steps are up-steps and whose last $r+s$ steps are
down-steps.

If $q\ge 1$, then, for every $N\ge 0$, $a_N$ counts tuples
\[
(P_0,P_1,\dots,P_q)
\]
with the following properties:
\begin{enumerate}[label=\textup{(\roman*)}]
\item $P_0\in \mathcal D_m(0,m-\rho-1;u_0)$;
\item $P_1\in \mathcal D_m(r,0;u_1)$, so $P_1$ is a Dyck path of height at most $m-1$ and
      semilength $m-1+u_1$ whose first $r$ steps are up-steps and whose last
      $m-1$ steps are down-steps;
\item for $2\le \nu\le q$, one has $P_\nu\in \mathcal D_m(0,0;u_\nu)$;
\item $u_0+u_1+\cdots+u_q=N$.
\end{enumerate}

\item[\textup{(c)}] More generally, let
\[
\xi=(m^t,r_1,\dots,r_d,1^s),
\qquad 1\le r_i\le m-1,
\]
and write
\[
r_1+\cdots+r_d+s=qm+\rho,\qquad q\in\mathbb Z_{\ge 0},\quad 0\le \rho<m.
\]
If $q\ge d$, then, for every $N\ge 0$, $a_N$ counts tuples
\[
(P_0,P_1,\dots,P_q)
\]
with the following properties:
\begin{enumerate}[label=\textup{(\roman*)}]
\item $P_0\in \mathcal D_m(0,m-\rho-1;u_0)$;
\item for $1\le i\le d$, one has $P_i\in \mathcal D_m(r_i,0;u_i)$;
\item for $d+1\le \nu\le q$, one has $P_\nu\in \mathcal D_m(0,0;u_\nu)$;
\item $u_0+u_1+\cdots+u_q=N$.
\end{enumerate}
\end{enumerate}
\end{theorem}
\medskip

The factorization hypothesis is restrictive: it requires the numerator product to decompose into exactly $k$ admissible factors corresponding to the $k$ denominator copies of $p_m(x)$. The explicit families below are precisely situations where such a decomposition occurs naturally. In particular, when the number of numerator factors becomes large relative to the denominator exponent $k$, one cannot generally expect positivity for the coefficients of the corresponding Chebyshev quotient. The positivity phenomenon in Theorem~\ref{thm:manifest} is also related to positive linearization formulas for orthogonal polynomials. Certain families, including the Chebyshev polynomials, admit expansions
\[
p_n(x)p_m(x)=\sum_k c_{m,n}^k p_k(x),
\qquad c_{m,n}^k\ge 0,
\]
which allow products in the numerator to be reorganized combinatorially; see \cite{deSCV}.

We now return from the quotient level to the representation-theoretic
multiplicities themselves.

\begin{corollary}\label{cor:manifest-mult}
Fix $N\ge 0$ and set
\[
n:=|\xi|-2N.
\]
If $n\geq 0$, then we have the following formula for the Demazure multiplicity
\[
V_n^{\xi\to m}(1).
\]
\begin{enumerate}
\item[\textup{(a)}] Let
\[
\xi=(m^t,1^s),
\]
and write
\[
s-2N=qm+\rho,\qquad q\in\mathbb Z,\quad 0\le \rho<m.
\]
Then we have
\[
V_{tm+s-2N}^{\xi\to m}(1)
=
[x^N]\frac{p_{m-\rho-1}(x)}{p_m(x)^{q+1}}.
\]
If $q<0$, then
\[
V_{tm+s-2N}^{\xi\to m}(1)=0.
\]
If $q\ge 0$, then $V_{tm+s-2N}^{\xi\to m}(1)$ is the number of tuples
\[
(P_0,P_1,\dots,P_q)
\]
for which
\[
P_0\in \mathcal D_m(0,m-\rho-1;u_0),
\qquad
P_\nu\in \mathcal D_m(0,0;u_\nu)\ \ (1\le \nu\le q),
\]
and
\[u_0+u_1+\cdots+u_q=N.\]

\item[\textup{(b)}] Let
\[
\xi=(m^t,r,1^s),\qquad 1\le r\le m-1,
\]
and write
\[
r+s-2N=qm+\rho,\qquad q\in\mathbb Z,\quad 0\le \rho<m.
\]
Then we have
\[
V_{tm+r+s-2N}^{\xi\to m}(1)
=
[x^N]\frac{p_r(x)p_{m-\rho-1}(x)}{p_m(x)^{q+1}}.
\]
If $q<0$, then
\[
V_{tm+r+s-2N}^{\xi\to m}(1)=0.
\]
If $q=0$ and $2N\le s$, then
\[
V_{tm+r+s-2N}^{\xi\to m}(1)=D_m(r,m-r-s+2N-1;N).
\]
If $q\ge 1$, then $V_{tm+r+s-2N}^{\xi\to m}(1)$ is the number of tuples
\[
(P_0,P_1,\dots,P_q)
\]
for which
\[
P_0\in \mathcal D_m(0,m-\rho-1;u_0),
\qquad
P_1\in \mathcal D_m(r,0;u_1),
\]
\[
P_\nu\in \mathcal D_m(0,0;u_\nu)\ \ (2\le \nu\le q),
\]
and
\[u_0+u_1+\cdots+u_q=N.\]
When $q=0$ and $2N>s$, no unsigned Dyck-path model is asserted here; the displayed coefficient
formula remains valid.

\item[\textup{(c)}] Let
\[
\xi=(m^t,r_1,\dots,r_d,1^s),\qquad 1\le r_i\le m-1,
\]
and write
\[
r_1+\cdots+r_d+s-2N=qm+\rho,\qquad q\in\mathbb Z,\quad 0\le \rho<m.
\]
Then
\[
V_{tm+r_1+\cdots+r_d+s-2N}^{\xi\to m}(1)
=
[x^N]\frac{p_{m-\rho-1}(x)\prod_{i=1}^d p_{r_i}(x)}{p_m(x)^{q+1}}.
\]
If $q<0$, then
\[
V_{tm+r_1+\cdots+r_d+s-2N}^{\xi\to m}(1)=0.
\]
If $q\ge d$, then $V_{tm+r_1+\cdots+r_d+s-2N}^{\xi\to m}(1)$ is the number of tuples
\[
(P_0,P_1,\dots,P_q)
\]
for which
\[
P_0\in \mathcal D_m(0,m-\rho-1;u_0),
\qquad
P_i\in \mathcal D_m(r_i,0;u_i)\ \ (1\le i\le d),
\]
\[
P_\nu\in \mathcal D_m(0,0;u_\nu)\ \ (d+1\le \nu\le q),
\]
and
\[u_0+u_1+\cdots+u_q=N.\]
If $0\le q<d$, no unsigned Dyck-path model is asserted here; the displayed coefficient formula
remains valid.
\end{enumerate}
\end{corollary}

Theorem~\ref{thm:manifest} is deliberately a quotient statement. The point is that the same
quotient can arise from different multiplicity problems once the target weight varies.
Corollary~\ref{cor:manifest-mult} records the corresponding direct formulas for numerical
Demazure multiplicities in the three families treated here.

The remainder of the paper is organized as follows. Section~2 proves the eventual-positivity
theorem by analyzing the roots of $p_m$. Section~3 interprets the basic generating-series factors
in terms of matchings and bounded strip walks, and Section~4 combines these interpretations to
prove the signed coefficient formula. Section~5 gives the unsigned bounded-Dyck-path quotient
families and then translates them back to direct multiplicity statements. Finally, in Section~\ref{sec:AxiomProver} we discuss how AxiomProver, an AI tool for mathematics research, autonomously produced and formalized the theorems in this paper.

\section{The root-theoretic proof of eventual positivity}

In this section we prove Theorem~\ref{thm:main}. The argument separates the Chebyshev quotient
into a cancellable polynomial part and a genuinely rational part, and then uses the smallest
positive root of $p_m$ to control the eventual sign of the coefficients. This is the only place
where we use the precise location of the roots of the Chebyshev-type polynomials.

We first record the structure of the roots of $p_m$ for $m\ge 2$.

\begin{lemma}\label{lem:roots}
Assume $m\ge 2$. The roots of $p_m(x)$ are
\[
\rho_j=\frac{1}{4\cos^2\!\left(\frac{j\pi}{m+1}\right)},
\qquad 1\le j\le \left\lfloor \frac m2\right\rfloor.
\]
These roots are simple, positive, and satisfy
\[
0<\rho_1<\rho_2<\cdots<\rho_{\lfloor m/2\rfloor}.
\]
In particular, $\rho_1$ is the unique root of $p_m$ having smallest modulus.
\end{lemma}

\begin{proof}
Set $x=z^2$.  The recurrence for $p_m$ is equivalent to the standard identity
\[
p_m(z^2)=z^m U_m\!\left(\frac{1}{2z}\right),
\]
where $U_m$ is the Chebyshev polynomial of the second kind.  The zeros of $U_m$ are
\(\cos(j\pi/(m+1))\), $1\le j\le m$, and
\[
U_m(\cos\theta)=\frac{\sin((m+1)\theta)}{\sin\theta}
\]
\cite[Ch.~2]{MasonHandscomb2002}.  Since $p_m(0)=1$, every root of $p_m$ is nonzero.  Thus a
zero of $p_m(z^2)$ satisfies
\[
\frac{1}{2z}=\cos\left(\frac{j\pi}{m+1}\right)
\]
for some $j$.  If $m$ is odd, the value $j=(m+1)/2$ gives the zero $0$ of $U_m$ and hence no
finite value of $z$.  The remaining indices occur in pairs $j$ and $m+1-j$, and these two values
of $z$ have opposite signs and hence the same value of $x=z^2$.  Therefore the distinct roots of
$p_m(x)$ are precisely
\[
\rho_j=\frac{1}{4\cos^2(j\pi/(m+1))},\qquad 1\le j\le \lfloor m/2\rfloor.
\]
The recurrence also gives $\deg p_m=\lfloor m/2\rfloor$, so the displayed distinct roots account
for all roots and are simple.

Finally,
\[
0<\frac{\pi}{m+1}<\frac{2\pi}{m+1}<\cdots\le
\frac{\lfloor m/2\rfloor\pi}{m+1}<\frac{\pi}{2},
\]
and $\cos\theta$ is strictly decreasing on $(0,\pi/2)$.  Hence the sequence
$\cos(j\pi/(m+1))$ is strictly decreasing and the sequence $\rho_j$ is strictly increasing.  Since
all roots are positive, $\rho_1$ is the unique root of smallest modulus.
\end{proof}

\begin{lemma}\label{lem:positive-at-rho1}
Assume $m\ge 2$. Let $\theta=\pi/(m+1)$ and $\rho_1=1/(4\cos^2\theta)$. If
$0\le r<m$, then
\[
p_r(\rho_1)>0.
\]
\end{lemma}

\begin{proof}
Using the Chebyshev formula, with the positive square root,
\[
p_r(x)=x^{r/2}U_r\!\left((2\sqrt{x})^{-1}\right),
\]
so
\[
p_r(\rho_1)=\rho_1^{r/2}U_r(\cos\theta)
=\rho_1^{r/2}\frac{\sin((r+1)\theta)}{\sin\theta}.
\]
Since $0<\theta<\pi$ and $0<(r+1)\theta\le m\theta=m\pi/(m+1)<\pi$ for $r<m$,
both denominator and numerator are positive. Hence $p_r(\rho_1)>0$.
\end{proof}

\begin{proof}[Proof of Theorem~\ref{thm:main}]
If $m=1$, then $\mu_0=0$, $p_0(x)=p_1(x)=1$, and every part of $\xi$ is equal to $1$.
Hence $p_\xi(x)=1$ and therefore $F(x)=1$. This proves part~(1).

Now assume $m\ge 2$. Write
\[
s:=\mu_1+1.
\]
Since $\xi_i\le m$ for all $i$, we may factor
\[
p_\xi(x)=p_m(x)^t H(x),
\]
where
\[
H(x)=\prod_{\xi_i<m} p_{\xi_i}(x)
\]
is a polynomial all of whose factors have index $<m$. Hence
\[
F(x)=\frac{p_{m-\mu_0-1}(x)H(x)}{p_m(x)^{s-t}}.
\]
Set
\[
G(x):=p_{m-\mu_0-1}(x)H(x).
\]
Every factor of $G$ is of the form $p_r$ with $0\le r<m$, because
$m-\mu_0-1\in\{0,1,\dots,m-1\}$. By Lemma~\ref{lem:positive-at-rho1}, each such factor
is positive at $\rho_1$, so
\[
G(\rho_1)>0.
\]

If $t\ge s$, then
\[
F(x)=p_m(x)^{t-s}G(x)
\]
is a polynomial, proving part~(2a).

Now assume $t\le \mu_1$, so that
\[
k:=s-t\ge 1.
\]
Then we have
\[
F(x)=\frac{G(x)}{p_m(x)^k}.
\]
By Lemma~\ref{lem:roots}, with $J:=\lfloor m/2\rfloor$, we have
\[
p_m(x)=\prod_{j=1}^J \left(1-\frac{x}{\rho_j}\right),
\]
so
\[
F(x)=\frac{G(x)}{\prod_{j=1}^J \left(1-\frac{x}{\rho_j}\right)^k}.
\]
Because $G(\rho_1)>0$, the function $F$ has a pole of order exactly $k$ at $x=\rho_1$.

If $J=1$, then $p_m(x)=1-x/\rho_1$, and so
\[
F(x)=\frac{G(x)}{\left(1-\frac{x}{\rho_1}\right)^k}.
\]
Write the Taylor expansion of $G$ at $x=\rho_1$ as
\[
G(x)=\sum_{j\ge 0} c_j\left(1-\frac{x}{\rho_1}\right)^j,
\]
where $c_0=G(\rho_1)>0$. Then
\[
F(x)=\sum_{j\ge 0} c_j\left(1-\frac{x}{\rho_1}\right)^{j-k}.
\]
Splitting off the terms $0\le j\le k-1$, we obtain
\[
F(x)=P(x)+\sum_{\ell=1}^k \frac{A_\ell}{\left(1-\frac{x}{\rho_1}\right)^\ell},
\]
where $P(x)$ is a polynomial and
\[
A_k=c_0=G(\rho_1)>0.
\]
The coefficient of $x^r$ in $(1-x/\rho_1)^{-\ell}$ is
\[
\binom{r+\ell-1}{\ell-1}\rho_1^{-r}.
\]
Hence the contribution from the pole at $\rho_1$ is
\[
M_r:=\sum_{\ell=1}^k A_\ell\binom{r+\ell-1}{\ell-1}\rho_1^{-r}
=Q(r)\rho_1^{-r},
\]
where $Q(r)$ is a polynomial of degree $k-1$ with leading coefficient
$A_k/(k-1)!>0$. Therefore $M_r>0$ for all sufficiently large $r$. Since the polynomial
part contributes only finitely many nonzero coefficients, we conclude that
$a_r>0$ for all sufficiently large $r$.

It remains to treat the case $J\ge 2$. Define
\[
\Phi(x):=\frac{G(x)}{\prod_{j=2}^J\left(1-\frac{x}{\rho_j}\right)^k}.
\]
Then $\Phi$ is analytic at $x=\rho_1$ and
\[
\Phi(\rho_1)=\frac{G(\rho_1)}{\prod_{j=2}^J\left(1-\frac{\rho_1}{\rho_j}\right)^k}>0,
\]
because $0<\rho_1<\rho_j$ for all $j\ge 2$. Thus, near $x=\rho_1$, we have
\[
F(x)=\frac{\Phi(x)}{\left(1-\frac{x}{\rho_1}\right)^k}.
\]
Therefore the partial fraction decomposition of $F$ contains a term
\[
\sum_{\ell=1}^k \frac{A_\ell}{\left(1-\frac{x}{\rho_1}\right)^\ell}
\]
with
\[
A_k=\lim_{x\to \rho_1}\left(1-\frac{x}{\rho_1}\right)^kF(x)=\Phi(\rho_1)>0.
\]
Hence, we may write
\[
F(x)=P(x)+\sum_{\ell=1}^k \frac{A_\ell}{\left(1-\frac{x}{\rho_1}\right)^\ell}
+\sum_{j=2}^J\sum_{\ell=1}^k \frac{B_{j,\ell}}{\left(1-\frac{x}{\rho_j}\right)^\ell},
\]
where $P(x)$ is a polynomial and the $B_{j,\ell}$ are constants.

The coefficient of $x^r$ in $(1-x/\rho)^{-\ell}$ is
\[
\binom{r+\ell-1}{\ell-1}\rho^{-r}.
\]
Hence the contribution from the pole at $\rho_1$ is again
\[
M_r:=\sum_{\ell=1}^k A_\ell\binom{r+\ell-1}{\ell-1}\rho_1^{-r}
=Q(r)\rho_1^{-r},
\]
where $Q(r)$ is a polynomial of degree $k-1$ whose leading coefficient is
$A_k/(k-1)!>0$. In particular, there exist constants $c>0$ and $R_1$ such that
\[
M_r\ge c\, r^{k-1}\rho_1^{-r}\qquad (r\ge R_1).
\]

Let $E_r$ denote the contribution from the polynomial part $P(x)$ and the poles
$\rho_j$ with $j\ge 2$. Since the polynomial part contributes only finitely many
nonzero coefficients, and since there are only finitely many pairs $(j,\ell)$ with
$j\ge 2$, there exist constants $C>0$ and $R_2$ such that
\[
|E_r|\le C\, r^{k-1}\rho_2^{-r}\qquad (r\ge R_2).
\]
Because $0<\rho_1<\rho_2$, we have $(\rho_1/\rho_2)^r\to 0$. Therefore,
for all sufficiently large $r$,
\[
|E_r|\le \frac12 M_r.
\]
For such $r$, we have
\[
a_r=M_r+E_r\ge M_r-|E_r|\ge \frac12 M_r>0.
\]
This proves part~(2b).
\end{proof}

\section{Walk and matching interpretations of the basic factors}

In this section we prove two standard facts: first, $1/p_m(x)$ and related quotients admit a
transfer-matrix interpretation in terms of strip walks; second, $p_r(x)$ is the matching
polynomial of a path graph after a simple change of variables. Both viewpoints are standard: see
Flajolet~\cite{Flajolet1980} for transfer-matrix/continued-fraction methods and Godsil~\cite{Godsil1993}
for matching polynomials.  

\begin{remark} The transfer-matrix interpretation underlying Corollary~\ref{cor:walkquotient} is classical and extends more generally to orthogonal polynomials and continued fractions; see, for example, Theorem~10.11.1 of Krattenthaler~\cite{KrattenthalerHEC}. Related determinant-based proofs appear already in Chapter~V of Viennot's monograph on orthogonal polynomials~\cite{Viennot}, and combinatorial proofs via heaps are discussed in Cigler and Krattenthaler~\cite{CK}. 
\end{remark}

\smallskip
Our purpose here is to specialize these methods to the Chebyshev quotients arising from Demazure multiplicities and to connect them with the signed and unsigned combinatorial models developed in later sections. More generally, quotient formulas of this type arise naturally in the
theory of orthogonal polynomials and bounded Motzkin path enumeration.
In Viennot's framework~\cite{Viennot}, the generating function for bounded
Motzkin paths can be expressed in terms of quotients of reciprocal
orthogonal polynomials. See also Stanton and Kim~\cite{KimStanton}
for a modern formulation of these results.
The Chebyshev case considered here corresponds to the specialization
$b_n=0$ and $\lambda_n=1$ (up to normalization), in which Motzkin paths
reduce to Dyck paths.

Let $A_m$ be the adjacency matrix of the path graph on vertices $0,1,\dots,m-1$:
\[
A_m=
\begin{pmatrix}
0&1\\
1&0&1\\
&1&0&1\\
&&\ddots&\ddots&\ddots\\
&&&1&0
\end{pmatrix}.
\]
For an indeterminate $s$, define
\[
K_m(s):=I-sA_m.
\]

We also set
\[
\Delta_m(s):=\det K_m(s),
\]
with the convention $\Delta_0(s)=1$.

\begin{lemma}\label{lem:continuant}
For every $m\ge 0$ one has
\[
\Delta_m(s)=p_m(s^2).
\]
\end{lemma}

\begin{proof}
The determinant $\Delta_m(s)$ satisfies the same recurrence as $p_m(s^2)$. Indeed, expanding
$\Delta_{m+1}(s)$ along the last row (or last column) gives
\[
\Delta_{m+1}(s)=\Delta_m(s)-s^2\Delta_{m-1}(s)\qquad (m\ge 1),
\]
with initial values
\[
\Delta_0(s)=1,\qquad \Delta_1(s)=1.
\]
Since $p_0(s^2)=p_1(s^2)=1$ and
\[
p_{m+1}(s^2)=p_m(s^2)-s^2p_{m-1}(s^2),
\]
it follows by induction on $m$ that $\Delta_m(s)=p_m(s^2)$ for all $m\ge 0$.
\end{proof}

\begin{proposition}\label{prop:inverse-entry}
Let $m\ge 1$ and $0\le a\le b\le m-1$. Then
\[
\bigl(K_m(s)^{-1}\bigr)_{a,b}
=s^{b-a}\frac{p_a(s^2)p_{m-1-b}(s^2)}{p_m(s^2)}.
\]
By symmetry the same formula holds for $a\ge b$ after interchanging $a$ and $b$.
\end{proposition}

\begin{proof}
By Cramer's rule,
\[
\bigl(K_m(s)^{-1}\bigr)_{a,b}=\frac{(-1)^{a+b}\det K_m(s)[b\mid a]}{\det K_m(s)},
\]
where $K_m(s)[b\mid a]$ denotes the matrix obtained by deleting row $b$ and column $a$.
When $a\le b$, the tridiagonal shape of $K_m(s)$ implies that this minor is block upper
triangular, with one $a\times a$ block equal to $K_a(s)$, one $(m-1-b)\times (m-1-b)$ block
 equal to $K_{m-1-b}(s)$, and a chain of $b-a$ off-diagonal $-s$ entries linking them.
Consequently,
\[
\det K_m(s)[b\mid a]=(-1)^{a+b}s^{b-a}\Delta_a(s)\Delta_{m-1-b}(s).
\]
Using Lemma~\ref{lem:continuant}, we obtain
\[
\bigl(K_m(s)^{-1}\bigr)_{a,b}
=s^{b-a}\frac{p_a(s^2)p_{m-1-b}(s^2)}{p_m(s^2)},
\]
as claimed.
\end{proof}

\begin{corollary}\label{cor:walkquotient}
Let $m\ge 1$ and $0\le a\le b\le m-1$. Then
\[
\frac{p_a(x)p_{m-1-b}(x)}{p_m(x)}
=\sum_{r\ge 0} w_{b-a+2r}^{(m)}(a,b)\,x^r.
\]
In particular, we have
\[
\frac{1}{p_m(x)}=\sum_{u\ge 0} B_m(u)x^u
\]
and
\[
\frac{p_{m-c-1}(x)}{p_m(x)}=\sum_{u\ge 0} w_{c+2u}^{(m)}(0,c)x^u\qquad (0\le c<m).
\]
\end{corollary}

\begin{proof}
Since the constant term of $K_m(s)$ is the identity matrix, the inverse exists as a formal
power series and the Neumann-series identity gives
\[
K_m(s)^{-1}=\sum_{L\ge 0}(sA_m)^L.
\]
The $(a,b)$-entry of $A_m^L$ counts walks of length $L$ from $a$ to $b$ on the path graph
with vertex set $\{0,1,\dots,m-1\}$, and those walks are exactly the strip walks defined in
Section~1. Therefore
\[
\bigl(K_m(s)^{-1}\bigr)_{a,b}=\sum_{L\ge 0} w_L^{(m)}(a,b) s^L.
\]
Combining this with Proposition~\ref{prop:inverse-entry} yields
\[
\sum_{L\ge 0} w_L^{(m)}(a,b)s^L
=s^{b-a}\frac{p_a(s^2)p_{m-1-b}(s^2)}{p_m(s^2)}.
\]
Now every strip walk from $a$ to $b$ has length congruent to $b-a\pmod 2$, so writing
$L=b-a+2r$ and substituting $x=s^2$ gives
\[
\frac{p_a(x)p_{m-1-b}(x)}{p_m(x)}
=\sum_{r\ge 0} w_{b-a+2r}^{(m)}(a,b)\,x^r.
\]
Taking $(a,b)=(0,m-1)$ gives the formula for $1/p_m(x)$, while taking
$(a,b)=(0,c)$ gives the second displayed identity.
\end{proof}
\begin{remark}
Corollary~\ref{cor:walkquotient} may be viewed as the Chebyshev specialization of more general bounded Motzkin-path generating functions associated with orthogonal polynomials; see Chapter~V of \cite{Viennot}.

\end{remark}
\begin{lemma}\label{lem:matchingcount}
For every $r\ge 0$, we have 
\[
p_r(x)=\sum_{j=0}^{\lfloor r/2\rfloor}(-1)^j\binom{r-j}{j}x^j.
\]
Equivalently, we have
\[
p_r(x)=\sum_{j\ge 0}(-1)^jm_r(j)x^j,
\]
where $m_r(j)$ is the number of $j$-edge matchings in the path graph $P_r$.
In particular,
\[
m_r(j)=\binom{r-j}{j}.
\]
\end{lemma}

\begin{proof}
The closed formula for $p_r(x)$ follows by an easy induction on $r$ from the recurrence
$p_{r+1}=p_r-xp_{r-1}$. The resulting interpretation as a matching polynomial for a path graph is
standard; see, for example, Godsil~\cite[Ch.~1]{Godsil1993}. For the matching interpretation, note that a $j$-edge matching in
$P_r$ may be specified by choosing indices
\[
1\le i_1< i_2<\cdots < i_j\le r-1
\]
with the additional condition $i_{\ell+1}\ge i_\ell+2$, because adjacent edges cannot both
belong to a matching. Setting
\[
y_\ell:=i_\ell-(\ell-1)
\]
produces a strictly increasing $j$-tuple
\[
1\le y_1<y_2<\cdots<y_j\le r-j.
\]
This construction is reversible, so the number of such matchings is
\[
\binom{r-j}{j}.
\]
Substituting this into the closed formula gives the stated matching-polynomial identity.
\end{proof}

\section{Proof of the combinatorial coefficient formula}

With the interpretations from the previous section in hand, the proof of Theorem~\ref{thm:comb}
is a direct coefficient extraction. We expand the numerator factors using the matching model, expand
each reciprocal factor using bounded strip walks, and then collect the coefficient of a fixed power
of $x$.

\begin{proof}[Proof of Theorem~\ref{thm:comb}]
By construction, we have
\[
F_{\xi,m,\mu}(x)=\frac{\prod_{i=0}^L p_{\alpha_i}(x)}{p_m(x)^k}.
\]
By Lemma~\ref{lem:matchingcount}, each factor $p_{\alpha_i}(x)$ has the expansion
\[
p_{\alpha_i}(x)=\sum_{j_i\ge 0}(-1)^{j_i}\binom{\alpha_i-j_i}{j_i}x^{j_i},
\]
where the binomial coefficient is automatically zero when $j_i>\lfloor \alpha_i/2\rfloor$.
By Corollary~\ref{cor:walkquotient}, we also have
\[
\frac{1}{p_m(x)}=\sum_{u\ge 0} B_m(u)x^u.
\]
Therefore, we have
\[
F_{\xi,m,\mu}(x)=\left(\prod_{i=0}^L \sum_{j_i\ge 0}(-1)^{j_i}\binom{\alpha_i-j_i}{j_i}x^{j_i}\right)
\left(\sum_{u\ge 0}B_m(u)x^u\right)^k.
\]
Expanding the product and collecting the coefficient of $x^r$ gives
\[
a_r=
\sum_{\substack{j_0,\dots,j_L\ge 0\\ u_1,\dots,u_k\ge 0\\
 j_0+\cdots+j_L+u_1+\cdots+u_k=r}}
(-1)^{j_0+\cdots+j_L}
\left(\prod_{i=0}^L \binom{\alpha_i-j_i}{j_i}\right)
\left(\prod_{\nu=1}^k B_m(u_\nu)\right),
\]
which is the first assertion.

For the equivalent signed description, we interpret each factor combinatorially. By
Lemma~\ref{lem:matchingcount}, the number of choices for a matching $M_i$ in $P_{\alpha_i}$ with
$|M_i|=j_i$ is exactly $\binom{\alpha_i-j_i}{j_i}$. By definition of $B_m(u_\nu)$, the number of
choices for a full-height strip walk $\gamma_\nu$ with excess $u_\nu$ is exactly $B_m(u_\nu)$.
Thus each summand in the displayed formula counts tuples
\[
(M_0,\dots,M_L,\gamma_1,\dots,\gamma_k)
\]
with total weight
\[
|M_0|+\cdots+|M_L|+e(\gamma_1)+\cdots+e(\gamma_k)=r,
\]
and sign
\[
(-1)^{|M_0|+\cdots+|M_L|}.
\]
Summing over all allowable weight vectors yields exactly the signed count stated in the theorem.
\end{proof}

\section{When the signed formula is manifestly positive}

In this section we prove Theorem~\ref{thm:manifest}. The key point is that the strip-walk
quotients from Corollary~\ref{cor:walkquotient} are equivalent to families of bounded Dyck
paths after a simple and explicit bijection.

\begin{proposition}\label{prop:dyck-bijection}
Fix $m\ge 1$ and integers $0\le a,b\le m-1$ with $a+b\le m-1$. For each $u\ge 0$, let
$\mathcal W_m(a,b;u)$ be the set of strip walks of height $m-1$ from $a$ to $m-1-b$ and length
\[
m-1-a-b+2u.
\]
For $\gamma=(h_0,h_1,\dots,h_L)\in \mathcal W_m(a,b;u)$, let $w(\gamma)$ denote the word in
$\{U,D\}$ obtained by recording an up-step when $h_{i+1}=h_i+1$ and a down-step when
$h_{i+1}=h_i-1$. Then the map
\[
\Phi_{a,b,u}:\mathcal W_m(a,b;u)\longrightarrow \mathcal D_m(a,b;u),
\qquad
\Phi_{a,b,u}(\gamma):=U^a\,w(\gamma)\,D^{m-1-b},
\]
is a bijection.
\end{proposition}

\begin{proof}
Let $\gamma=(h_0,h_1,\dots,h_L)\in \mathcal W_m(a,b;u)$. Since $\gamma$ starts at height $a$,
ends at height $m-1-b$, and stays between heights $0$ and $m-1$, the word $w(\gamma)$ describes a
lattice path that begins at height $a$, ends at height $m-1-b$, and remains inside the strip
$0\le y\le m-1$. After prepending $a$ up-steps and appending $m-1-b$ down-steps, the resulting
path never goes below the $x$-axis and never rises above height $m-1$. Thus
$\Phi_{a,b,u}(\gamma)$ is a Dyck path of height at most $m-1$.

Its total length is
\[
a+(m-1-a-b+2u)+(m-1-b)=2(m-1-b+u),
\]
so its semilength is $m-1-b+u$. By construction, its first $a$ steps are up-steps and its last
$m-1-b$ steps are down-steps. Hence $\Phi_{a,b,u}(\gamma)$ belongs to $\mathcal D_m(a,b;u)$.

Conversely, let $P\in \mathcal D_m(a,b;u)$. By definition, the first $a$ steps of $P$ are
up-steps and the last $m-1-b$ steps are down-steps. Delete those initial and terminal blocks. The
remaining middle segment starts at height $a$, ends at height $m-1-b$, and stays in the strip
$0\le y\le m-1$. Its length is
\[
2(m-1-b+u)-a-(m-1-b)=m-1-a-b+2u,
\]
so it lies in $\mathcal W_m(a,b;u)$. This construction is plainly inverse to
$\Phi_{a,b,u}$, so $\Phi_{a,b,u}$ is a bijection.
\end{proof}

\begin{proof}[Proof of Theorem~\ref{thm:manifest}]
Assume first that we are given integers $(a_\nu,b_\nu)$ for $1\le \nu\le k$ with
\[
0\le a_\nu,b_\nu\le m-1,\qquad a_\nu+b_\nu\le m-1,
\]
and
\[
\prod_{i=0}^L p_{\alpha_i}(x)=\prod_{\nu=1}^k p_{a_\nu}(x)p_{b_\nu}(x).
\]
Then, we have
\[
F_{\xi,m,\mu}(x)=\prod_{\nu=1}^k \frac{p_{a_\nu}(x)p_{b_\nu}(x)}{p_m(x)}.
\]
Fix $\nu$. Since $a_\nu+b_\nu\le m-1$, Corollary~\ref{cor:walkquotient} with
\[
a=a_\nu,\qquad b=m-1-b_\nu
\]
gives
\[
\frac{p_{a_\nu}(x)p_{b_\nu}(x)}{p_m(x)}
=
\sum_{u\ge 0} w_{m-1-a_\nu-b_\nu+2u}^{(m)}(a_\nu,m-1-b_\nu)\,x^u.
\]
By Proposition~\ref{prop:dyck-bijection}, the coefficient on the right is exactly
$D_m(a_\nu,b_\nu;u)$. Therefore
\[
\frac{p_{a_\nu}(x)p_{b_\nu}(x)}{p_m(x)}
=
\sum_{u\ge 0} D_m(a_\nu,b_\nu;u)\,x^u.
\]
Multiplying these $k$ identities proves the product formula
\[
F_{\xi,m,\mu}(x)=\prod_{\nu=1}^k\left(\sum_{u\ge 0} D_m(a_\nu,b_\nu;u)\,x^u\right).
\]
Now expand the product and collect the coefficient of $x^r$. One obtains
\[
a_r=
\sum_{\substack{u_1,\dots,u_k\ge 0\\ u_1+\cdots+u_k=r}}
\prod_{\nu=1}^k D_m(a_\nu,b_\nu;u_\nu).
\]
This is precisely the number of $k$-tuples $(P_1,\dots,P_k)$ such that
$P_\nu\in \mathcal D_m(a_\nu,b_\nu;u_\nu)$ for each $\nu$ and
$u_1+\cdots+u_k=r$. In particular, all coefficients $a_r$ are nonnegative.

We now prove the three explicit families.

For part~\textup{(a)}, let $\xi=(m^t,1^s)$ and write $s=qm+\rho$ with $q\in\mathbb Z_{\ge 0}$ and $0\le \rho<m$.
Then
\[
\mu=s+tm=(t+q)m+\rho,
\]
so
\[
\mu_1=t+q,\qquad \mu_0=\rho.
\]
Since $p_1(x)=1$, we have $p_\xi(x)=p_m(x)^t$, and therefore
\[
F_{\xi,m,\mu}(x)=\frac{p_{m-\rho-1}(x)}{p_m(x)^{q+1}}
=
\frac{p_{m-\rho-1}(x)}{p_m(x)}
\left(\frac{1}{p_m(x)}\right)^q.
\]
The first factor is the case $(a,b)=(0,m-\rho-1)$ of the general criterion, while each factor
$1/p_m(x)$ is the case $(a,b)=(0,0)$. Hence
\[
\frac{p_{m-\rho-1}(x)}{p_m(x)}=\sum_{u\ge 0} D_m(0,m-\rho-1;u)\,x^u
\]
and
\[
\frac{1}{p_m(x)}=\sum_{u\ge 0} D_m(0,0;u)\,x^u.
\]
Multiplying these series and extracting the coefficient of $x^N$ gives exactly the tuple count
stated in part~\textup{(a)}.

For part~\textup{(b)}, let $\xi=(m^t,r,1^s)$ with $1\le r\le m-1$, and write
$r+s=qm+\rho$ with $q\in\mathbb Z_{\ge 0}$ and $0\le \rho<m$.

Assume first that $q=0$, so $r+s<m$ and $\rho=r+s$. Then
\[
\mu=r+s+tm=tm+\rho,
\qquad \mu_1=t,
\qquad \mu_0=\rho=r+s.
\]
Again $p_1(x)=1$, so
\[
F_{\xi,m,\mu}(x)=\frac{p_{m-r-s-1}(x)p_r(x)}{p_m(x)}.
\]
This is the single-factor case of the general criterion with
\[
a=r,\qquad b=m-r-s-1.
\]
Therefore, we have
\[
F_{\xi,m,\mu}(x)=\sum_{u\ge 0} D_m(r,m-r-s-1;u)\,x^u,
\]
which is exactly the claimed Dyck-path description.

Now assume $q\ge 1$. Then
\[
\mu=r+s+tm=(t+q)m+\rho,
\qquad \mu_1=t+q,
\qquad \mu_0=\rho,
\]
and therefore
\[
F_{\xi,m,\mu}(x)=\frac{p_{m-\rho-1}(x)p_r(x)}{p_m(x)^{q+1}}
=
\frac{p_{m-\rho-1}(x)}{p_m(x)}\cdot
\frac{p_r(x)}{p_m(x)}\cdot
\left(\frac{1}{p_m(x)}\right)^{q-1}.
\]
These factors correspond respectively to $(a,b)=(0,m-\rho-1)$, $(a,b)=(r,0)$, and
$(a,b)=(0,0)$. Hence, we have
\[
\frac{p_{m-\rho-1}(x)}{p_m(x)}=\sum_{u\ge 0} D_m(0,m-\rho-1;u)\,x^u,
\]
\[
\frac{p_r(x)}{p_m(x)}=\sum_{u\ge 0} D_m(r,0;u)\,x^u,
\]
\[
\frac{1}{p_m(x)}=\sum_{u\ge 0} D_m(0,0;u)\,x^u.
\]
Multiplying these series and extracting the coefficient of $x^N$ gives the tuple count stated in
part~\textup{(b)}.

For part~\textup{(c)}, let $\xi=(m^t,r_1,\dots,r_d,1^s)$ with $1\le r_i\le m-1$, and write
\[
r_1+\cdots+r_d+s=qm+\rho,\qquad q\in\mathbb Z_{\ge 0},\quad 0\le \rho<m.
\]
Then we have
\[
\mu=r_1+\cdots+r_d+s+tm=(t+q)m+\rho,
\]
so
\[
\mu_1=t+q,\qquad \mu_0=\rho.
\]
Because $p_1(x)=1$, we have
\[
F_{\xi,m,\mu}(x)=\frac{p_{m-\rho-1}(x)\prod_{i=1}^d p_{r_i}(x)}{p_m(x)^{q+1}}.
\]
If $q\ge d$, then we may rewrite this as
\[
F_{\xi,m,\mu}(x)=\frac{p_{m-\rho-1}(x)}{p_m(x)}
\prod_{i=1}^d \frac{p_{r_i}(x)}{p_m(x)}
\left(\frac{1}{p_m(x)}\right)^{q-d}.
\]
These factors correspond to $(a,b)=(0,m-\rho-1)$, to $(a,b)=(r_i,0)$ for $1\le i\le d$,
and to $(a,b)=(0,0)$ for the remaining factors. Therefore
\[
\frac{p_{m-\rho-1}(x)}{p_m(x)}=\sum_{u\ge 0} D_m(0,m-\rho-1;u)\,x^u,
\]
\[
\frac{p_{r_i}(x)}{p_m(x)}=\sum_{u\ge 0} D_m(r_i,0;u)\,x^u
\qquad (1\le i\le d),
\]
\[
\frac{1}{p_m(x)}=\sum_{u\ge 0} D_m(0,0;u)\,x^u.
\]
Multiplying these series and extracting the coefficient of $x^N$ gives the tuple count stated in
part~\textup{(c)}.
\end{proof}

\begin{proof}[Proof of Corollary~\ref{cor:manifest-mult}]
When $n\ge 0$, the displayed coefficient formula in each case is obtained by applying
Proposition~\ref{prop:cheb-quotient} with
\[
n=|\xi|-2N.
\]
If $n<0$, the multiplicity is interpreted as zero as stated; in the subcases below this agrees
with the displayed coefficient, because the same degree estimates show that the relevant
coefficient vanishes when the Euclidean quotient $q$ is negative. We then rewrite the resulting
quotient using the Euclidean division displayed in the statement. We also use the elementary bound
$\deg p_a\le a/2$, which follows immediately from the recurrence defining $p_a$.

For part~\textup{(a)}, write
\[
s-2N=qm+\rho,\qquad q\in\mathbb Z,\quad 0\le \rho<m.
\]
Then we have
\[
tm+s-2N=(t+q)m+\rho,
\]
so Proposition~\ref{prop:cheb-quotient} gives
\[
V_{tm+s-2N}^{\xi\to m}(1)
=
[x^N]\frac{p_{m-\rho-1}(x)}{p_m(x)^{q+1}}.
\]
If $q\ge 0$, this is exactly the quotient appearing in Theorem~\ref{thm:manifest}(a), so the
stated bounded-Dyck-path interpretation follows by extracting the coefficient of $x^N$.
If $q<0$, write $q=-h-1$ with $h\ge 0$. Then
\[
\frac{p_{m-\rho-1}(x)}{p_m(x)^{q+1}}=p_{m-\rho-1}(x)p_m(x)^h
\]
is a polynomial of degree at most
\[
\frac{m-\rho-1+hm}{2}=\frac{(h+1)m-\rho-1}{2}<N,
\]
because
\[
2N=s+(h+1)m-\rho.
\]
Hence, the coefficient of $x^N$ is zero.

For part~\textup{(b)}, write
\[
r+s-2N=qm+\rho,\qquad q\in\mathbb Z,\quad 0\le \rho<m.
\]
Then we have
\[
tm+r+s-2N=(t+q)m+\rho,
\]
so Proposition~\ref{prop:cheb-quotient} gives
\[
V_{tm+r+s-2N}^{\xi\to m}(1)
=
[x^N]\frac{p_r(x)p_{m-\rho-1}(x)}{p_m(x)^{q+1}}.
\]
If $q\ge 1$, we factor the quotient as
\[
\frac{p_{m-\rho-1}(x)}{p_m(x)}\cdot \frac{p_r(x)}{p_m(x)}\cdot
\left(\frac1{p_m(x)}\right)^{q-1}
\]
and apply the general criterion in Theorem~\ref{thm:manifest} with the pairs
\[
(0,m-\rho-1),\qquad (r,0),\qquad (0,0),\dots,(0,0).
\]
This yields the stated tuple model.

Now assume $q=0$. Then
\[
\rho=r+s-2N,
\qquad
m-\rho-1=m-r-s+2N-1.
\]
If $2N\le s$, then $m-r-s+2N-1=m-\rho-1$ lies between $0$ and $m-1$, and
\[
r+(m-r-s+2N-1)=m-s+2N-1\le m-1.
\]
Thus Theorem~\ref{thm:manifest} applies with the single pair
\[
(a_1,b_1)=(r,m-r-s+2N-1),
\]
and therefore
\[
\frac{p_r(x)p_{m-\rho-1}(x)}{p_m(x)}=
\sum_{u\ge 0} D_m(r,m-r-s+2N-1;u)x^u.
\]
Extracting the coefficient of $x^N$ gives the stated formula.
If instead $q<0$, write $q=-h-1$ with $h\ge 0$. Then
\[
\frac{p_r(x)p_{m-\rho-1}(x)}{p_m(x)^{q+1}}=p_r(x)p_{m-\rho-1}(x)p_m(x)^h
\]
is a polynomial of degree at most
\[
\frac{r+(m-\rho-1)+hm}{2}=\frac{r+(h+1)m-\rho-1}{2}<N,
\]
because
\[
2N=r+s+(h+1)m-\rho.
\]
Hence the coefficient of $x^N$ is zero.
The only remaining subcase is $q=0$ and $2N>s$, and there we make no unsigned Dyck-path claim.

For part~\textup{(c)}, write
\[
r_1+\cdots+r_d+s-2N=qm+\rho,
\qquad q\in\mathbb Z,\quad 0\le \rho<m.
\]
Then Proposition~\ref{prop:cheb-quotient} gives
\[
V_{tm+r_1+\cdots+r_d+s-2N}^{\xi\to m}(1)
=
[x^N]\frac{p_{m-\rho-1}(x)\prod_{i=1}^d p_{r_i}(x)}{p_m(x)^{q+1}}.
\]
If $q\ge d$, we rewrite this quotient as
\[
\frac{p_{m-\rho-1}(x)}{p_m(x)}
\prod_{i=1}^d \frac{p_{r_i}(x)}{p_m(x)}
\left(\frac1{p_m(x)}\right)^{q-d}
\]
and apply Theorem~\ref{thm:manifest} with the pairs
\[
(0,m-\rho-1),\qquad (r_1,0),\dots,(r_d,0),\qquad (0,0),\dots,(0,0).
\]
This yields the stated tuple model.
If $q<0$, write again $q=-h-1$ with $h\ge 0$. Then the quotient is the polynomial
\[
p_{m-\rho-1}(x)\prod_{i=1}^d p_{r_i}(x)p_m(x)^h,
\]
of degree at most
\[
\frac{(m-\rho-1)+r_1+\cdots+r_d+hm}{2}
=\frac{r_1+\cdots+r_d+(h+1)m-\rho-1}{2}<N,
\]
because
\[
2N=r_1+\cdots+r_d+s+(h+1)m-\rho.
\]
Hence the coefficient of $x^N$ is zero. This proves the corollary.
\end{proof}

\begin{remark}
We emphasize that an alternative combinatorial model for the fat-hook families considered here was obtained in \cite{BiswalKus2021} in terms of admissible Dyck paths. That construction is different from the present one, which is designed to arise directly from the Chebyshev quotient.

The transfer-matrix and continued-fraction interpretations underlying Corollary~\ref{cor:walkquotient} are classical. General forms for quotients associated with orthogonal polynomials appear, for example, in Krattenthaler~\cite[Theorem~10.11.1]{KrattenthalerHEC}; see also Chapter~V of Viennot's monograph~\cite{Viennot}. Determinant-style proofs of this type are standard, while alternative combinatorial proofs using heaps are discussed in \cite{CK}.

In \cite{BiswalKus2021}, a co-major index statistic on admissible Dyck paths was shown to capture the full graded multiplicities. In contrast, for the combinatorial model introduced in this paper, the appropriate statistic that realizes the full graded multiplicities is not yet known. Determining such a statistic remains an interesting open problem.
\end{remark}

\section{Appendix: Autonomous proof production and formalization}\label{sec:AxiomProver}

At Axiom Math, we are developing AxiomProver, an AI system for mathematical research based on autoformalization. As a test case, we gave AxiomProver two tasks\footnote{The formalization of Theorem~\ref{thm:main} was completed before the authors had finalized the content of Theorems~\ref{thm:comb} and ~\ref{thm:manifest}, which explains the existence of two separate tasks.} of autoformalizing Theorem~\ref{thm:main},~\ref{thm:comb}  and~\ref{thm:manifest} in this paper, offering examples of AI assistance in mathematical research.  This appendix is separate from the rest of the manuscript. The prose exposition of this paper, including this appendix, was written without the use of AI. 
AxiomProver runs Lean,
 an interactive theorem prover and a functional programming language built on dependent type theory, designed to provide a rigorous computational framework for validating mathematical proofs \cite{Lean}. 
This appendix is included not only to record that formal proofs were produced,
but also to clarify the division of labor between the human-written
mathematics and the AI-generated formalizations.  The
mathematical statements supplied to AxiomProver were written in natural
language by the authors. AxiomProver's role was to convert these statements
into Lean and to construct machine-checkable formal proofs. Thus, in the tasks
described below, AxiomProver was not being asked to choose the main theorems
or to decide the final organization of the paper; it was being asked to
produce complete formal proofs of the supplied statements.

 Autoformalization involves automatically and autonomously  converting natural-language mathematics into machine-verifiable formal language. Lean files are created to pass type checkers, while the natural language papers aim to communicate ideas to readers.
Given problems in natural or formal language, AxiomProver attempts to generate a complete formal proof. When it succeeds, the system produces two files:
\begin{itemize}
  \item \texttt{problem.lean}, which formalizes the problem statement if a formal problem statement is absent;
  \item \texttt{solution.lean}, which represents a complete proof in a formal language.
\end{itemize}

We now describe the two tasks in more detail:

\noindent
{\bf Task 1 (Theorem~\ref{thm:main}).} We give the following files as input for AxiomProver:
\begin{itemize}
    \item \texttt{task.md} contains the informal statement of Theorem~\ref{thm:main}.
    \item \texttt{.environment} specifies the version used is 4.28.0.
\end{itemize}
\smallskip
\noindent{\bf Task 2 (Theorem~\ref{thm:comb} and~\ref{thm:manifest}).} We give the following files as input for AxiomProver:
\begin{itemize}
    \item \texttt{source.tex} contains the informal statement of Theorem~\ref{thm:main},~\ref{thm:comb} and~\ref{thm:manifest}.
    \item \texttt{theorem1.lean} is a verbatim copy of the output from Task 1.
    \item \texttt{task.md} instructs AxiomProver to formalize Theorem~\ref{thm:comb} and~\ref{thm:manifest} and informs AxiomProver that \texttt{theorem1.lean} is a formalization of Theorem~\ref{thm:main}.
    \item \texttt{.environment} specifies the version used is 4.28.0.
\end{itemize}

The two tasks should therefore be viewed as sequential, but not as a case in
which the first task already contained the full content of the second. Task~1
concerned only the eventual-positivity dichotomy of Theorem~\ref{thm:main}. Its output,
\texttt{theorem1.lean}, was supplied to Task~2 as a verified formal ingredient
and as context for the notation and hypotheses already formalized. Task~2 then
addressed the separate combinatorial assertions in Theorems~\ref{thm:comb} and~\ref{thm:manifest}. In
particular, the formalization of Theorem~\ref{thm:manifest} was not merely extracted from the
formal proof of Theorem~\ref{thm:main}; it required AxiomProver to formalize additional
objects and arguments, including the matching and bounded-walk interpretations
and the quotient identities used in the unsigned Dyck-path cases.

At the time Task~2 was run, the detailed natural-language statements of
Theorems~\ref{thm:comb} and~\ref{thm:manifest} had already been formulated by the authors and were
included in \texttt{source.tex}. AxiomProver was therefore not used to discover
the statement of Theorem~\ref{thm:manifest}. Rather, its contribution was to produce a
machine-checkable Lean formalization and proof of the already supplied
statement, using the formalization of Theorem~\ref{thm:main} from Task~1 as an available
input.

In both cases, AxiomProver autonomously produced Lean files that were accepted
by the Lean type checker. Here ``autonomously'' means that, after receiving the
specified input files, AxiomProver generated the corresponding
\texttt{problem.lean} and \texttt{solution.lean} files without further
mathematical intervention in the proof script. The authors did not hand-write
or repair the Lean proofs line by line. The resulting formal proofs were then
used as machine-checkable certificates for the stated theorems. The relevant
files are posted in the following repository\footnote{The version used is
4.28.0. Compatibility with other versions is not guaranteed due to the evolving
nature of the Lean~4 compiler and its core libraries.}:
\[
\texttt{https://github.com/AxiomMath/Biswal}.
\]

At first glance, the proofs generated by AxiomProver do not resemble the narrative outlined in this paper. Converting a Lean file into a proof understandable by humans is challenging because Lean is designed as code for a type-checker, not as a reader-friendly explanation. It makes all the ``obvious'' bookkeeping explicit, such as rewrite steps, coercions, side conditions, and case splits, and tends to follow the most convenient lemmas and tactics for the library, rather than the most clear conceptual route. A mathematician can usually condense this significantly by relying on shared historical context, standard arguments, and informal identifications that Lean cannot assume. As a result, writing a paper from Lean files is not just about reformatting. The authors must understand the formal script, reconstruct the underlying ideas, and then translate the code into a narrative that emphasizes the key insights while safely omitting routine details.


\medskip


\section*{Acknowledgments}

The authors are grateful to Christian Krattenthaler and Dennis Stanton
for helpful comments and references regarding quotients of orthogonal
polynomials, bounded lattice path models, and positivity phenomena.
Their observations helped improve the exposition and historical context
of this paper.

\end{document}